\documentclass[oneside,leqno,11pt]{amsart}
\pdfoutput=1
\usepackage{amsthm}
\usepackage{amsmath, amssymb, amsfonts}
\usepackage{graphicx}
\usepackage[textwidth=16cm,textheight=20cm]{geometry}
\usepackage{hyperref}
\usepackage{bbm}
\usepackage{tikz}
\usepackage{pgfplots}
\pgfplotsset{width=7cm,compat=1.9}
\usepackage{color,epsfig}

\newtheorem{thm}[equation]{Theorem}

\newtheorem{lem}[equation]{Lemma}

\newtheorem{prop}[equation]{Proposition}
\theoremstyle{definition}

\numberwithin{equation}{section}

\def\P{{\mathbb{P}}}
\def\TT{{\mathcal{T}}}
\def\T{{\mathbb{T}}}
\def\R{{\mathbb{R}}}
\def\E{{\mathbb{E}}}
\def\S{{\mathbb{S}}}
\def\X{{\mathbb{X}}}
\def\M{{\mathbb{M}}}
\def\8{\infty}
\renewcommand{\a}{\alpha}
\renewcommand{\d}{\delta}
\newcommand{\g}{\gamma}
\newcommand{\eps}{\varepsilon}

\newcommand{\wt}{\widetilde}
\newcommand{\ov}{\overline}

\renewcommand{\le}{\leqslant}\renewcommand{\leq}{\leqslant}
\renewcommand{\ge}{\geqslant}\renewcommand{\geq}{\geqslant}
\renewcommand{\setminus}{\smallsetminus}

\begin{document}

\title{ Large deviation estimates for branching random walks}
\author[D. Buraczewski, M. Ma\'slanka]
{Dariusz Buraczewski, Mariusz Ma\'slanka}
\address{D. Buraczewski, M. Ma\'slanka\\ Instytut Matematyczny\\ Uniwersytet Wroclawski\\ 50-384 Wroclaw\\
pl. Grunwaldzki 2/4\\ Poland}
\email{dbura@math.uni.wroc.pl\\ maslanka@math.uni.wroc.pl}

\thanks{The  research was partially supported by the National Science Center, Poland (Sonata Bis, grant number DEC-2014/14/E/ST1/00588)}

\date{\today}


\maketitle
\begin{abstract}
We consider the branching random walk drifting  to $-\infty$ and we 
investigate large deviations-type estimates for the first passage
time. We prove the corresponding law of large numbers and the central limit theorem.
\end{abstract}

\maketitle

\section{Introduction}
\subsection{Branching random walk}
We consider a discrete-time one-dimensional branching random walk. An initial particle is located at the origin. At time 1 it gives birth to $N$ new particles, and then dies. Each of the particles is positioned randomly on the real line according to the distribution of the point process $\mathcal{L}$.
Next, at time 2 all the individuals produce independently their own children and die. All the particles in the second generation are located according to the same point process, with respect to the positions of their parents. This procedure
perpetuates itself. The resulting system is called a branching random walk. It can be represented as an infinite tree $\mathcal{T} =\bigcup_{k \geq 0} \{1,2,...,N \}^k$, where $o=\emptyset$ denotes the initial ancestor and  word $\g$ of length $n$ corresponds to individuals in the $n$th generation. With every node $\g$ one can associate a real random variable $X_\g$ representing its displacement according to the parent. Then, the position of $\g$ is given by $S_\g$, the sum of all weights on the path from $o$ to $\g$. The collection of positions $\{S_\g\}_{\g\in{\mathcal T}}$ forms the branching random walk.

In this paper we are interested in behavior of
\begin{equation*}
M_n = \max_{|\g|=n} S_\g,
\end{equation*}
 the maximal position of the branching random walk after $n$ steps. (Usually one considers the minimal position, however it is sufficient to replace the point process $\mathcal{L}$, by $-\mathcal{L}$.) Its properties are coded in the Laplace transform of the point process $\mathcal{L}$
$$
\psi(s) = \E\bigg[  \sum_{|\g|=1} e^{s X_\g}  \bigg].
$$ The starting point of our considerations is the law of large numbers proved in a sequence of papers by Hammersley \cite{hammersley}, Kingman \cite{kingman} and Biggins \cite{biggins},
$$
\lim_{n\to\8} \frac {M_n}n = \rho^* \qquad \mbox{a.s.}
$$ for $\rho^* = \inf_{s>0} \frac{\log \psi(s)}{s}$. Precise large deviations  were proved by Rouault \cite{R}. Last years, in a number of papers, the problem of fluctuations, i.e. behavior of $M_n - n\rho^*$, was considered. Addario-Berry, Reed \cite{addario} and Hu, Shi \cite{Hu:Shi} exhibited a logarithmic correction and next A\"id\'ekon \cite{A} proved that the normalized fluctuations converge in law to some random shift of a Gumbel variable. We refer the reader to a recent book by Shi \cite{Shi} for an overview of the related results.

\medskip

In this paper  we consider the  case when  the branching random walk drifts to $-\8$, i.e.
 $\rho^* <0$ and $\psi(s)<1$ for some $s>0$. Some of the trajectories can still exceed a large level $u$ and apart from knowing the probability of large deviations, we  point out the moment when it arises. More precisely we consider the first
passage time
\begin{equation}\label{eq: tu}
  \tau_u = \inf_n \{ M_n > u\}.
\end{equation}
Then, conditioning on the event $\{\tau_u<\8\}$ we prove limit theorems related to $\tau_u$: the law of large numbers, the central limit theorem and large deviations. All these results are formulated as Theorems \ref{thm: tuclt} and \ref{thm: tu-ld}.

\subsection{Main results}
Let $\mathcal{T} =\bigcup_{k \geq 0} \{1,2,...,N \}^k$ be an infinite tree, where $\{1,2,...,N\}^0 = \{ \varnothing \}$. For $\g = (i_1, . . . , i_n)\in \mathcal{T}$ we denote the length of $\g$ by $|\g| = n$, by $\g i$ we denote the vertex $(i_1, i_2, . . . , i_n, i)$ and we put $\g|_k = (i_1,\ldots,i_k)$ for $k\le |\g|$. Thus, every node $\g\in \mathcal{T}$ has $N$ children of the form $\g i$.
With tree $\mathcal{T}$ we associate an i.i.d sequence $\{X_\g\}_{\g \in \mathcal{T}}$ with the law of some generic random variable $X$. We assume that the distribution of $X$ is non-lattice, i.e. it
is not supported on any of the sets $a\mathbb{Z}+b$, $a > 0$. For any $\g = (i_1, . . . , i_n) \in \mathcal{T}$ define $$S_\g = X_{(i_1)} + X_{(i_1,i_2)} + ... + X_{(i_1,i_2,...,i_n)}  = X_{\g|_1}+\cdots + X_{\g|_n},$$ then $S_\g$ denotes the position of the particle represented by $\gamma$.

In this paper we are interested in behavior of $M_n = \max_{ |\g|=n} S_{\g}$, i.e. maximal position of the branching random walk after $n$ steps. Its properties are determined by the moment and cumulant generating functions
\begin{align*}
 \psi(s) & = \E\left[\sum_{i=1}^{N} e^{s X_i} \right] = N \E\left[e^{s X_1} \right] =: N \lambda(s),\\
 \Psi(s) &= \log \psi(s), \Lambda(s) = \log \lambda(s).
\end{align*}
We define two parameters
$$\a_\8 = \sup\{\a:\; \Psi(\a)<\8\},\qquad  \rho_\8 = \sup_{\a<\a_\8}\{\Psi'(\a)\}.
$$
 Our rate function is just the convex conjugate (the Fenchel-Legendre transform) of $\Psi$ and is defined by
$$
\Psi^*(x) = \sup_{s\in \R}\{sx - \Psi(s)\}, \quad x\in\R.
$$ Its various properties can be found in Dembo, Zeitouni \cite{DZ}. We will often use that given $\a<\a_\8$ and $\rho=\Psi'(\a)$
\begin{equation}\label{eq:ren24}
\Psi^*(\rho) = \rho \a - \Psi(\a).
\end{equation}
We assume that $\rho^*<0$ and we study
behavior of the first passage time $\tau_u$ defined in \eqref{eq: tu}. It is known (see Jelenkovic, Olvera-Cravioto \cite{JO}) that if $\psi(\a_0)=1$ for some $\a_0\in(0,\a_\8)$, then
\begin{equation}\label{eq: jo}
\P\big[\tau_u <\8\big] = \P\big[ \sup_{\g\in\mathcal{T}} S_\g > u \big]\sim c_+ e^{-\a_0 u}.
\end{equation}
Here we obtain limit theorems for $\tau_u$ conditioned on $\{\tau_u < \8\}$.
\begin{thm}[Law of large numbers and Central limit theorem]\label{thm: tuclt} If  $\psi(\a_0)=1$ for some $\a_0 \in (0,\a_\8)$
 and the distribution of $X$ is non-lattice, then
%\begin{equation}\label{eq: 51}
$$\frac{\tau_u}{u}\; \bigg| \tau_u <\8 \stackrel{\P}{\longrightarrow} \frac 1{\rho_0}
$$%\end{equation}
and
%\begin{equation}\label{eq: 52}
$$\frac{\tau_u - u/\rho_0}{ \sigma_0 \rho_0^{-3/2} \sqrt{  u}}\; \bigg| \tau_u <\8 \stackrel{d}{\longrightarrow} N(0,1),
$$%\end{equation}
where
$\rho_0 = \Psi'(\a_0)$, $\sigma_0 = \Psi''(\a_0)$.\footnote{The symbol  $\stackrel{\P}\longrightarrow$ (resp. $\stackrel{d}\longrightarrow$) denotes convergence in probability (resp. in ditribution).}
\end{thm}

\begin{thm}[Large deviations]
\label{thm: tu-ld}
Assume that the distribution of $X$ is non-lattice,   $\rho\in (0,\rho_\8)$   and let $\a$ satisfies $\Psi'(\a)=\rho$.

If $\psi(\a)>1$ (in particular $\rho>\rho_0$), then
\begin{equation}\label{eq: ren41}
\P\bigg[  \frac{\tau_u}{u} < \frac 1{\rho} \bigg] \sim \frac{C_1 \psi(\a)^{-\Theta(u)}}{\sqrt u   }\; {e^{-\frac{\Psi^*(\rho)}{\rho} u}},
\end{equation} for $\Theta(u) = u/\rho - \lfloor u/\rho\rfloor$.

If $\psi(\a)<1$ (in particular $\rho<\rho_0$ if $\rho_0$ is well defined), then
\begin{equation}\label{eq: ren42}
\P\bigg[  \frac{\tau_u}{u} > \frac 1{\rho} \bigg] \sim \frac{C_2 \psi(\a)^{-\Theta(u)}}{\sqrt u   }\; {e^{-\frac{\Psi^*(\rho)}{\rho} u}}.
\end{equation}
Moreover
\begin{equation}\label{eq: ren43}
\P\bigg[ \tau_u =  \bigg\lfloor \frac u{\rho}\bigg\rfloor \bigg] \sim \frac{C_3 \psi(\a)^{-\Theta(u)}}{\sqrt u   } \;{e^{-\frac{\Psi^*(\rho)}{\rho} u}}.
\end{equation}

\end{thm}

The above results generalize classical estimates for random walks (von Bahr \cite{von}, Lalley \cite{Lalley}, H\"oglund \cite{hoglund}, see also Buraczewski, Ma\'slanka \cite{BM}). Similar results were proved recently for analogous first passage times in the context of random difference equation (see \cite{BCDZ, BDZ}) and
branching process in random environment (see Afanasyev \cite{af1,af2} and Buraczewski, Dyszewski \cite{BD}).

\section{Auxiliary definitions and results}

\subsection{Some further definitions}
Here we collect some additional definitions that will be needed in the proof. If $\g_1 = i_1^1i_2^1..i_{n_1}^1\in \TT$ and $\g_2 = i_1^2i_2^2..i_{n_2}^2\in \TT$ then we write $\g_1 \g_2 = i_1^1i_2^1..i_{n_1}^1i_1^2i_2^2..i_{n_2}^2 $
for the element of $\TT$ obtained by juxtaposition. In particular $\g\emptyset = \emptyset \g = \g$. We write $\g' < \g$ if $\g'$ is a proper prefix of $\g$, i.e. $\g' = (i_1, .., i_k)$ for some $k < n$.  For two vertices $\g_1$ and $\g_2$ we denote by $\g_0 = \g_1\wedge \g_2$ the longest common subsequence of $\g_1$ and $\g_2$, i.e. the maximal $\g_0$ such that both $\g_0\le \g_1$ and $\g_0\le \g_2$. Moreover, we write $\g' \leq \g$ if $\g' < \g$ or $\g' = \g$.

We will consider different subsets of the tree and maximums over these subsets, namely we
define
\begin{align*}
T_{\g,n} &= \{ \g\g':\; |\g'| = n\},\\
\T_{\g,n} &= \bigcup_{k=1}^n T_{\g,k} =  \{ \g\g':\; 0< |\g'|\le n\}.\\
\end{align*}
Thus $T_{\g,n}$ is the set of progeny of $\g$ in generation $|\g|+n$ and $\T_{\g,n}$ consists of all descendants of $\g$ up to time $|\g|+n$.
If $\g = o$ we omit the vertex in the subscript and just write
$$
T_{n} = T_{o,n} =  \{ \g':\; |\g'| = n\},
$$ that is $T_n$ is exactly the $n$th generation.
Moreover we define
\begin{align*}
M_{\g} &= \max_{n\le |\g|} S_{\g|n},\\
M_{A} &= \max_{ \g\in A} S_{\g}, \qquad \mbox {for any } A \subset \mathcal{T},\\
\overline M_n &= \max_{ |\g|\le n} S_{\g} = \max_{k\le n} M_k.\\
\end{align*}

To simplify our notation we consider a generic random walk $\S = \X_1+\cdots + \X_n$ where $\X_i$ are i.i.d. copies of $X$. Then for any $\g \in T_n$, $S_\g$ and $\S_n$ have the same law. Define $\M_n = \max_{1\le k \le n} \S_k$ and $\ov \tau_u = \inf\{ k:\; \S_k >u \}$. We define also $\S_i^n = \S_n - \S_{n-i} = \X_{n-i+1}+\cdots+\X_n$.

\subsection{Large deviations for classical random walks}

The proofs of our results strongly base on large deviation estimates of the usual random walk. We recall here
  Petrov's Theorem \cite{Petrov}. Similar result was proved independently by
Bahadur and Rao \cite{BR}, but since we need uniform estimates we follow here \cite{Petrov}. We state here slightly more general version that will be needed in subsequent sections.

\begin{lem}
\label{lem:petrov2}
Assume that the law of $X$ is non-lattice. Choose $\rho$ such that $\E X <\rho < \rho_\8$,  let $\a$ be the parameter satisfying $\Psi'(\a)=\rho$ and let
$\sigma = \Psi''(\a)$.
Then
$$
\P\big[ \S_{n+j_n} > n(\rho+\eps_n) \big] \sim \frac 1{\a \sigma \sqrt{2\pi n}} e^{-n\Psi^*(\rho)} N^{-(n+j_n)} \psi(\a)^{j_n} e^{-\a n \eps_n} e^{-\frac{(-\rho j_n + n\eps_n)^2}{2\sigma^2n}}
$$
as $n\to\8$ uniformly with respect to
$$
\sqrt n |\eps_n| \le \delta_n \to 0 \quad \mbox{and} \quad  |j_n| \le b \sqrt{n\log n}.
$$
\end{lem}
\begin{proof}
  First let us recall Theorem 2 in \cite{Petrov} saying that
\begin{align*}
\P\big[ \S_{n} > n(\rho+\eps_n) \big] &\sim \frac 1{\a \sigma \sqrt{2\pi n}} e^{n\big( \Lambda(\a) - \a (\rho +\eps_n) - \frac{\eps_n^2}{2\sigma^2} (1+ |\eps_n|) \big)}\\
& =  \frac 1{\a \sigma \sqrt{2\pi n}} e^{-n\Psi^*(\rho)}N^{-n} e^{-\a \eps_n n}
e^{-  \frac{n\eps_n^2}{2\sigma^2} (1+ |\eps_n|)}
\end{align*}
as $n\to \8$, uniformly for $|\eps_n| < \d_n\to 0$. Replacing in the formula above $(n,\eps_n)$ by $\big(n+j_n, -\frac{\rho j_n}{n+j_n} + \frac{n\eps_n}{n+j_n}\big)$ we obtain
\begin{align*}
\P\big[ \S_{n+j_n} > n(\rho+\eps_n) \big] &=
\P\bigg[ \S_{n+j_n} > (n+j_n)\bigg( \rho- \frac{\rho j_n}{n+j_n} + \frac{n\eps_n}{n+j_n} \bigg) \bigg] \\
&\sim
 \frac 1{\a \sigma \sqrt{2\pi (n+j_n)}} e^{-(n+j_n)\Psi^*(\rho)} N^{-(n+j_n)}  e^{-\a(-\rho j_n + n\eps_n)}
  e^{-\frac{(-\rho j_n + n\eps_n)^2}{2\sigma^2(n+j_n)}}\\
&\sim  \frac 1{\a \sigma \sqrt{2\pi n}} e^{-n\Psi^*(\rho)} N^{-(n+j_n)} \psi(\a)^{j_n} e^{-\a_n \eps_n} e^{-\frac{(-\rho j_n + n\eps_n)^2}{2\sigma^2n}}.
\end{align*}
The last line follows from the fact, that in the considered region $  e^{\frac{(-\rho j_n + n\eps_n)^2}{2\sigma^2(n+j_n)}}  e^{-\frac{(-\rho j_n + n\eps_n)^2}{2\sigma^2n}}$
converges uniformly to 1 as $n$ tends to $\8$.
\end{proof}

\subsection{Large deviations for first passage times for random walks}

The following Lemma was proved in \cite{hoglund} and \cite{BM}

\begin{lem}
\label{lem:1}
If the distribution of $X$ is non-lattice, $\E X <0$ and $\rho = \Psi'(\a) > 0$ for some $\rho, \a$, then
$$
\P[\ov\tau_{n\rho}=n] = \P\big[\M_{n-1} \le n\rho, \S_n > n\rho \big]
 \sim  \frac{c(\rho) e^{-\Psi^*(\rho) n}}{\sqrt n N^{n}}.
$$
\end{lem}
\begin{proof} We follow here the notation in \cite{BM} where the results were formulated in
terms of the function $\Lambda$ (see Theorem 2.3 in \cite{BM}). However,
since $\Psi'(\a)=\Lambda'(\a)$ and $\Psi(\a)=\log N +\Lambda(\a)$,  we have
$$
\P[\ov\tau_{n\rho}=n] \sim   \frac{c e^{-\a n\rho} e^{\frac{\Lambda(\a)}{\Lambda'(\a)}n\rho }}{\sqrt n}
=  \frac{c e^{-\a n\rho} e^{\frac{\Psi(\a)}{\Psi'(\a)}n\rho }e^{-{n \log N}}}{\sqrt n}
=\frac{c e^{-\Psi^*(\rho) n}}{\sqrt n N^{n}}.
$$
\end{proof}

However, for our purpose a uniform version of this result is needed.

\begin{lem}
  \label{lem: bmu}
If the distribution of $X$ is non-lattice, $\E X <0$ and $\rho = \Psi'(\a) > 0$ for some $\rho, \a$, then
  \begin{equation}\label{eq:bmu1}
\P[\ov \tau_{n\rho+a_n}=n] \sim  \frac{c(\rho) e^{-\a a_n} e^{-\Psi^*(\rho) n}}{\sqrt n N^{n}}
  \end{equation}
  uniformly as $n\to\8$ for any sequence $a_n$ such that $-K\log n < a_n < K\log n$ for some constant $K$.

  If $\rho=\rho_0$ and $\a=\a_0$, then
  \begin{equation}\label{eq:bmu2}
\P[\ov \tau_{n\rho_0}=n+j_n] \sim  \frac{c(\rho_0)  }{\sqrt n N^{n+j_n}} \; e^{-\Psi^*(\rho_0) n} e^{-\frac{\rho_0^2 j_n^2}{2\sigma_0^2 n}}
  \end{equation}
  uniformly as $n\to\8$ for any sequence $j_n$ such that $|j_n| \le b \sqrt{n\log n}$ for some constant $b$.
\end{lem}
We omit the proof since it is very close to the arguments presented in \cite{BM}. Only a few places require some minor corrections. Below in the proof we use similar techniques and we  will explain in details the role of uniformity.

\section{Uniform large deviations for $\tau_u$}

The aim of this section is to prove a stronger version of \eqref{eq: ren43}  with perturbed parameters.
In the next section we will use this result to deduce Theorems \ref{thm: tuclt} and \ref{thm: tu-ld}.
To simplify our presentation we state both results separately.

\begin{prop} \label{prop: cz1} For any $\rho\in (0,\rho_\8)$ and $\a$ such that $\Psi'(\a)=\rho$,
\begin{equation}\label{eq: ren44}
\P\big[\tau_{n\rho + a_n} = n \big] =   \P\big[ \ov M_{n-1} \le n\rho + {a_n}, M_n > n\rho + {a_n} \big] \sim C(\rho) \frac{e^{-\a {a_n}} e^{- \Psi^*(\rho) n}}{\sqrt{n}}
\end{equation}
uniformly as $n\to\8$ for any sequence ${a_n}$ such that $-K \log n < {a_n} < K \log n$ for some large constant $K$.

If $\rho = \rho_0$ and $\a=\a_0$, then
\begin{equation}\label{eq: ren44'}
\P\big[\tau_{n\rho_0} = n+j_n \big]  = \P\big[ \ov M_{n+j_n-1} \le n\rho_0, M_{n+j_n} > n\rho_0  \big] \sim \frac{  C(\rho_0) }{\sqrt{n}} e^{- \a_0 n\rho_0}e^{-\frac{\rho_0^2 j_n^2}{2\sigma_0^2 n}}
\end{equation}
uniformly as $n\to\8$ for any $|j_n|< b \sqrt{n\log n}$.
\end{prop}

\subsection{Lower and upper estimates}

\begin{lem}\label{prop:2}
There are constants $c_1,c_2>0$ such that
\begin{equation}\label{eq: czw1}
c_1 \le  \sqrt n  e^{\a a_n} e^{\Psi^*(\rho)  n}  \P\big[ \tau_{n\rho+a_n}=n  \big] \le c_2
\end{equation}
uniformly as $n\to\8$ for any $a_n$ such that $-K \log n < a_n < K \log n$ for some large constant $K$.

Moreover, if $\rho = \rho_0$ and $\a=\a_0$, then for $c_3,c_4>0$
\begin{equation}\label{eq: czw3}
c_3 \le  \sqrt{n}   e^{ \a_0 \rho_0 n}e^{\frac{\rho_0^2 j_n^2}{2\sigma_0^2 n}}
\P\big[ \tau_{n\rho_0}=n+j_n  \big] \le c_4
\end{equation}
uniformly as $n\to\8$ for any $|j_n|< b \sqrt{n\log n}$.
\end{lem}
\begin{proof} We prove here only the second part of the Lemma, i.e. \eqref{eq: czw3}. The first part \eqref{eq: czw1} can be proved exactly in the same way with obvious changes.

\medskip

\noindent
{\sc Step 1. Upper estimates.}
 For $\g\in T_{n+j_n}$ define
 $$C_{\gamma} = \{ M_{\gamma|_{n+j_n-1}} \leq n\rho_0, S_{\gamma} > n\rho_0 \}$$
 to be the event that the path from $o$ to $\g$ exceeds $n\rho_0$ for the first time exactly at $\g$. Then
 $$
\big\{\tau_{n\rho_0}=n +j_n \big\} \subset \bigcup_{\g\in T_{n+j_n}} C_\g
 $$ and Lemma \ref{lem: bmu} implies immediately
 \begin{align*}
 \P\big[\tau_{n\rho_0}  =n +j_n \big] &\le \P\bigg[ \bigcup_{\g\in T_{n+j_n}} C_\g \bigg] \le \sum_{\g\in T_{n+j_n}} \P[C_\g]
\\ &\le \frac {c|T_{n+j_n}|}{\sqrt n N^{n+j_n}} {c e^{-\a_0 \rho_0 n } e^{- \frac{\rho_0^2 j_n^2}{2\sigma_0^2 n}  } } \\ &=
  \frac{c}{\sqrt n} e^{-\a_0 \rho_0 n } e^{- \frac{\rho_0^2 j_n^2}{2\sigma_0^2 n}  },
 \end{align*}
 which gives the upper bound \eqref{eq: czw3}.

\medskip

{\sc Step 2. Lower estimates.} Unfortunately, to obtain the lower bound  the way we need to pass is quite long and tedious.
The idea is quite simple and bases on arguments presented in \cite{BDZ2}. We choose a sparse subset $U$ of $T_{n+j_n}$,  prove that the corresponding trajectories are almost independent and we apply Lemma \ref{lem: bmu}. For a large constant $L$ (its value will be specified below) we define
$$U = \{\g\in T_{n+j_n}:\; \g=\g|_{n+j_n-L}\underbrace{(1,...,1)}_\text{L times}\}. $$
This is the set of elements of $T_{n+j_n}$, whose last $L$ indices are 1's.
Note that $|U|=N^{n+j_n-L}$.
Then our aim is to prove that if we choose large $L$ and restrict our attention to elements of the set $U$,
  for some positive constants $d_i$

$$
\P\big[\tau_{n\rho_0}=n +j_n \big] \ge d_1 \P\bigg[\bigcup_{\g\in U} C_\g \bigg]\ge d_2 \sum_{\g\in U}\P[C_\g] \ge
\frac {d_3}{\sqrt n N^L} \cdot   e^{-\a_0 \rho_0 n } e^{- \frac{\rho_0^2 j_n^2}{2\sigma_0^2 n}  }.
$$
Since $L$ is fixed, the constant $d_3/N^L$, although very small, is strictly positive.
Thus,  $ \bigcup_{\g\in U} C_\g$ is not a subset of $\{\tau_{n\rho_0} =n+j_n\}$ nevertheless  one can efficiently compare probabilities of both sets. First,  we need to
modify further  the set   $C_\g$. The details are as follows.

For $\g\in U$ denote
$$W_{\gamma} = \{\gamma'\in \T_{\g|_{n+j_n-L}, L-1} \mbox{ and } \g' \mbox{ is not an ancestor of } \g\}.$$
Thus $W_\g$ is the set of all elements of $\T_{\g|_{n+j_n-L}, L-1}$ except those lying on the path between $\g|_{n+j_n-L}$ and $\g$. We
define  $$\wt{C}_{\gamma} = C_{\gamma} \cap \{  X_{\gamma'} < 0 \mbox{ for all } {\gamma' \in W_{\gamma}} \}.$$
Observe
%\begin{align*}
$$  \{\tau_{n\rho_0}=n +j_n \}  \supset \bigg( \bigcup_{\g\in U} C_\g  \bigg) \cap \{ M_{T_{n+j_n-1}}\le n\rho_0 \}
 % = \bigcup_{\g\in U} \Big( C_\g   \cap \{ M_{T_{n_u-1}}\le u\}\Big)\\
  \supset \bigcup_{\g\in U} \Big( \wt C_\g   \cap \{ M_{T_{n+j_n-1}}\le n\rho_0 \}\Big)
$$ %%\end{align*}
Our proof consists of two steps. First we prove that (see step 2A below)
\begin{equation}\label{eq:w2}
  \P \bigg[  \bigcup_{\gamma \in U}\wt{C}_{\gamma} \bigg] > \frac{\eps_1(L)}{\sqrt n } e^{-\a_0 \rho_0 n } e^{- \frac{\rho_0^2 j_n^2}{2\sigma_0^2 n}  } ,
\end{equation} for $\eps_1(L) = c p_L N^{-L}$
and next (see step 2B below)
\begin{equation}\label{eq:25}
  \P \bigg[  \bigg( \bigcup_{\gamma \in U}\wt{C}_{\gamma}\bigg) \cap \big\{ M_{T_{n+j_n-1}} >n\rho_0  \big\} \bigg] <  \frac{\eps_1(L)}{2\sqrt n } \cdot
  e^{-\a_0 \rho_0 n } e^{- \frac{\rho_0^2 j_n^2}{2\sigma_0^2 n}  }.
\end{equation}
Then combining \eqref{eq:w2} and \eqref{eq:25} we obtain
\begin{align*}
  \P\big[\tau_{n\rho_0} = n +j_n\big] & \ge   \P \bigg[  \bigg( \bigcup_{\gamma \in U}\wt{C}_{\gamma}\bigg) \cap \big\{ M_{T_{n+j_n-1}} \le n\rho_0  \big\} \bigg]\\
  &=   \P \bigg[ \bigcup_{\gamma \in U}\wt{C}_{\gamma} \bigg] -
    \P \bigg[  \bigg( \bigcup_{\gamma \in U}\wt{C}_{\gamma}\bigg) \cap \big\{ M_{T_{n+j_n-1}} >n\rho_0  \big\} \bigg]\\
    & > \frac{\eps_1(L)}{2\sqrt n} \cdot e^{-\a_0 \rho_0 n } e^{- \frac{\rho_0^2 j_n^2}{2\sigma_0^2 n}  }
\end{align*}

\noindent
{\sc Step 2A. Proof of \eqref{eq:w2}.} We observe that by the inclusion-exclusion formula
\begin{equation}\label{eq:41}
  \P \bigg[  \bigcup_{\gamma \in U}\wt{C}_{\gamma} \bigg]
  \ge \sum_{\g\in U} \P\big(\wt C_\g\big) - \sum_{\g,\g'\in U, \g\not=\g'}\P\big( \wt C_\g \cap \wt C_{\g'} \big).
\end{equation}
Since $S_\g$ and $X_{\g'}$ for $\g'\in W_\g$ are independent
$$
\P(\wt C_\g) = p_L\P(C_\g),
$$
for $$ p_L = \P \Big[  X_{\gamma'} < 0 \mbox{ for all } {\gamma' \in W_{\gamma}}  \Big] = \P \left[X_{1} < 0 \right]^{|W_{\gamma}|} > 0.$$
We prove that the sum of intersections in \eqref{eq:41} is of smaller order. For this purpose we group all $\g'\in U$ depending on the level of the first common ancestor with $\g$. Given $k\le n+j_n-L+1$ there are $N^{n+j_n-L-k}$ vertices $\g'\in U$ such that $|\g\wedge \g'|=k$. Therefore, applying the Markov inequality (with some $\beta$ such that
$\psi(\beta)<1$), we have
\begin{align*}
  \sum_{\g,\g'\in U,\g\not=\g'}&\P\big( \wt C_\g \cap \wt C_{\g'} \big)  = p_L \sum_{\g,\g'\in U, \g\not=\g'}\P\big( \wt C_\g \cap C_{\g'} \big) \\
&\leq p_L \sum_{\gamma \in U} \sum_{k=0}^{n+j_n - L - 1} \P\Big[ \wt C_\g \cap \big\{ S_{\g|_k} \le n\rho_0  \mbox{ and } S_{\g'}>n\rho_0 \mbox{ for some }\g'\in U \mbox{ with }|\g\wedge \g'|=k   \big\} \Big]\\
&\leq p_L \sum_{\gamma \in U} \sum_{k=0}^{n+j_n - L - 1} N^{n+j_n - L - k} \P\big[\wt{C}_{\gamma}\big]\P\left[ \S_{n+j_n - k} > 0 \right]\\
&\leq p_L \sum_{\gamma \in U} \P\big[\wt{C}_{\gamma}\big] N^{-L} \sum_{k=0}^{n +j_n - L - 1} \psi(\beta)^{n +j_n- k}\\
&\leq p_L C  \lambda(\beta)^{L} \sum_{\gamma \in U} \P\big[\wt{C}_{\gamma}\big]. %= \epsilon_L \sum_{\gamma \in U} \P\left[\wt{C}_{\gamma}\right].
\end{align*}
Therefore by Lemma \ref{lem: bmu}, for large $L$ such that $\eps_L = C p_L \lambda(\beta)^L < 1/2$,
\begin{align*}
    \P \bigg[  \bigcup_{\gamma \in U}\wt{C}_{\gamma} \bigg] &  \ge (1-\eps_L)  \sum_{\gamma \in U} \P \big[ \wt{C}_{\gamma} \big]\\
    & \ge \frac{p_L|U|}{2{\sqrt n N^{n+j_n}}}   e^{-\a_0 \rho_0 n } e^{- \frac{\rho_0^2 j_n^2}{2\sigma_0^2 n}  } \\
%    & \ge  \frac{\eps(L) e^{-\ov \a u}}{\sqrt u N^{n_u-L}}\\
     &\ge \frac{\eps_1(L)}{\sqrt n} \cdot  e^{-\a_0 \rho_0 n } e^{- \frac{\rho_0^2 j_n^2}{2\sigma_0^2 n}  },
\end{align*}
%for $\eps_1(L) = c p_L N^{-L}$,
which proves \eqref{eq:w2}.

\medskip

\noindent
{\sc Step 2B. Proof of  \eqref{eq:25}.}
%We proceed here similarly as in the proof of Theorem \ref{thm: ld} (compare Step 3B).
Pick any $\a_{\text{min}} <\beta < \a_0$. Define $\eta_1 = \frac{\lambda(\beta)}{\lambda(\alpha_0)} = {\psi(\beta)} < 1$ and  $\eta_2 = \eta_1 e^{\delta(\a_0 - \beta)}$. Take $\delta$ such that $\eta_2 < 1$.
We define
\begin{equation}
\label{eq:02}
V_{\g,L} =  \{S_{\gamma|_k} \leq n\rho_0 - \delta (n+j_n  - k) \text{ for all } k \leq n +j_n - L \}
\end{equation}
for $\g\in U$ and estimate
\begin{equation}\label{eq:29}
\begin{split}
\P\bigg[\bigg(\bigcup_{\g\in U} \wt C_\g\bigg) & \cap \big\{ M_{T_{n+j_n-1}} >n\rho_0    \big\}
\bigg] \le N^{n+j_n-L} \P\big[ \wt C_\g \cap \big\{ M_{T_{n+j_n-1}} >n\rho_0   \big\} \big]\\
 &\le N^{n+j_n-L} \Big( \P\big[ \wt C_\g \cap V_{\g,L}^c \big] +
\P\big[ \wt C_\g \cap  V_{\g,L} \cap  \big\{ M_{T_{n+j_n-1}} >n\rho_0   \big\} \big]\Big).
\end{split}
\end{equation}
In the next steps we estimate separately the probabilities above.

\medskip
\noindent

{\sc Step 2B$_1$.} We prove that for any $L>0$
\begin{equation}\label{eq:26}
  \P\big[ \wt C_\g \cap V_{\g,L}^c \big]  = p_L \P\big[ C_\g \cap V_{\g,L}^c \big] \le \frac{ C p_L \kappa_1^L}{\sqrt n N^{n}}
   e^{-\a_0 \rho_0 n } e^{- \frac{\rho_0^2 j_n^2}{2\sigma_0^2 n}  }
\end{equation}
for some $\kappa_1<1$ and $C>0$. 

We write
\begin{align*}
  \P\big[ C_\g & \cap V_{\g,L}^c \big]   \le
\P\left[\S_{n+j_n} > u,\S_{k} > n\rho_0 - \delta (n +j_n - k ) \text{ for some } k \leq n +j_n - L  \right] \\
&\leq  \sum_{k=1}^{n +j_n - L} \P\left[\S_{n+j_n} > n\rho_0, \S_{k} > n\rho_0 - \delta (n +j_n - k ) \right]\\
&\leq \sum_{k=1}^{n +j_n - L} \sum_{m=0}^{\infty} \P\left[\S_{k} + \S_{n+j_n - k}^{n+j_n} > n\rho_0,
 m< \S_k -  n\rho_0 + \delta (n +j_n - k) \le m + 1 \right]\\
& \leq \sum_{k=1}^{n +j_n - L} \sum_{m=0}^{\infty}\P\left[\S_{k} > n\rho_0 - \delta (n +j_n - k ) + m \right] \P\left[\S_{n+j_n-k}^{n+j_n} > \delta (n +j_n - k) -(m + 1) \right]\\
& = \sum_{k=1}^{n+j_n - L} \sum_{m=0}^{\infty}\P\left[\S_{k} > n\rho_0- \delta (n+j_n - k ) + m \right] \P\left[\S_{n+j_n-k} > \delta (n+j_n - k) -(m + 1) \right].
\end{align*}

%Repeating the same calculations as in \eqref{eq: wp1} we prove
%%\begin{align*}
%$$  \P\big[ C_\g \cap V_{\g,L}^c \big]   \le \!\!\!
%%\P\left[S_{n} > u,S_{k} > n\rho+{a_n} - \delta (n - k - L/4) \text{ for some } k \leq n - L  \right] \\
%%&\leq  \sum_{k=1}^{n - L} \P\left[S_{n} > n\rho+{a_n}, S_{k} > n\rho+{a_n} - \delta (n - k - L/4) \right]\\
%%&\leq \sum_{k=1}^{n - L} \sum_{m=0}^{\infty} \P\left[S_{k} + S_{n - k}^{n} > n\rho+{a_n},
%% m< S_k -  (n\rho+{a_n}) + \delta (n - k - L/4) \le m + 1 \right]\\
%%& \leq \sum_{k=1}^{n - L} \sum_{m=0}^{\infty}\P\left[S_{k} > n\rho+{a_n} - \delta (n - k - L/4) + m \right] \P\left[S_{n-k}^{n} > \delta (n - k - L/4) -(m + 1) \right]\\
%%&
%  \sum_{k=1}^{n +j_n - L}\!\!\! \sum_{m=0}^{\infty}\P\left[\S_{k}\!\! > n\rho_0\!-\! \delta (n + j_n - k) + m \right] \P\left[\S_{n+j_n-k} > \delta (n +j_n- k) -(m + 1) %\right].
%$$%\end{align*}

Fix $K > 0$ such that $\eta_2^{K \log n} < \frac 1{n^{1+d}}$ for $d = \frac{\rho_0^2 b}{2\sigma_0^2}$.  We divide the sum into two terms and study both of them independently.

Case 1: We first consider indices $k \leq n+j_n - K \log n$. By the  Markov  inequality
\begin{align*}
 \sum_{k = 1}^{n+j_n - K \log n}& \sum_{m=0}^{\infty}\P\left[\S_{k} > n\rho_0 - \delta (n +j_n\! -\! k) +m \right] \P\left[\S_{n+j_n-k} >\delta (n+j_n \! -\! k)-( m + 1)\right]\\
& \leq\!\! \sum_{k = 1}^{n+j_n - K \log n}\!\!\lambda(\a_0)^k e^{-\a_0 \rho_0 n} e^{\a_0 \delta(n+j_n - k)} \lambda(\beta)^{n+j_n -k} e^{-\beta \delta (n+j_n - k)} \sum_{m=0}^{\infty} e^{-\a_0 m}  e^{\beta (m+1)}\\
& \leq   \frac{ C e^{-{\a_0}\rho_0 n}}{N^{n+j_n}} \sum_{k = 1}^{n+j_n - K \log n} \eta_2^{n+j_n-k} \le \frac{C e^{-\a_0 \rho_0 n}}{N^{n+j_n}} \eta_2^{K \log n}\\
& = o\bigg( \frac{    e^{-\a_0 \rho_0 n }
}{\sqrt n N^{n+j_n}}\; e^{- \frac{\rho_0^2 j_n^2}{2\sigma_0^2 n}  } \bigg).
\end{align*}

Case 2: Consider $n+j_n - K \log n < k \leq n +j_n - L $. By Lemma \ref{lem:petrov2} and the Markov  inequality
\begin{align*}
& \!\!\!\!\!\!\sum_{k = n+j_n - K \log n+1}^{n+j_n - L} \sum_{m=0}^{\infty}\P\left[\S_{k} > \rho n_0  - \delta (n\!+j_n -\! k) +m \right] \P\left[\S_{n+j_n-k} >\delta (n\!+j_n -\! k)-( m + 1)\right]\\
& \le  C\!\!\!\!\!\!\!\!\!\!\!\!\!\!\sum_{k = n+j_n - K \log n+1}^{n+j_n - L} \sum_{m < n^{1/4}}\!\!\! \frac 1{\sqrt n N^k} e^{-\a_0 \rho_0 n} e^{\a_0 \d (n+j_n-k)} e^{-\a_0 m}
e^{-\frac{(\rho_0 (k-n) +\d (n+j_n-k) - m )^2
}{2\sigma_0^2 n}}\!\!\!  e^{-\beta \d (n+j_n-k)} e^{\beta m} \lambda(\beta)^{n+j_n - k}\\
& + C \sum_{k = n+j_n - K \log n+1}^{n+j_n - L} \sum_{m \ge n^{1/4}} \frac 1{ N^k} e^{-\a_0 \rho_0 n} e^{\a_0 \d (n+j_n-k)} e^{-\a_0 m}
\cdot e^{-\beta \d (n+j_n-k)} e^{\beta m} \lambda(\beta)^{n+j_n - k}\\
&\le \frac{ C e^{-\a_0 \rho_0 n}}{N^{n+j_n}} \sum_{k = n+j_n - K \log n+1}^{n+j_n - L} \eta_2^{n+j_n-k}
\bigg( \frac { e^{-\frac{\rho_0^2 j_n^2}{ 2\sigma_0^2 n}}
}{\sqrt n}  \sum_{m< n^{1/4}} e^{-(\a_0 - \beta)m} +  \sum_{m\ge n^{1/4}} e^{-(\a_0 - \beta)m}
\bigg)\\
&\le \frac{ C e^{-\a_0 \rho_0 n}}{N^{n+j_n}} \cdot \eta_2^L
\bigg( \frac { e^{-\frac{\rho_0^2 j_n^2}{ 2\sigma_0^2 n}}
}{\sqrt n} +  e^{-(\a_0-\beta)n^{1/4}} \bigg)
\le \frac{ C \eta_2^L} {\sqrt n N^{n+j_n}}         e^{-\a_0 \rho_0 n} e^{-\frac{\rho_0^2 j_n^2}{ 2\sigma_0^2 n}}.\\
\end{align*}
This proves \eqref{eq:26}.

\medskip

\noindent
{\sc Step 2B$_2$.} Now we are going to prove
\begin{equation}\label{eq:27}
  \P\big[ \wt C_\g \cap  V_{\g,L} \cap  \big\{ M_{T_{n+j_n-1}} >n\rho_0   \big\} \big] \le \frac{ C p_L \kappa_2^L}{\sqrt n N^{n+j_n}}   e^{-\a_0 \rho_0 n} e^{-\frac{\rho_0^2 j_n^2}{ 2\sigma_0^2 n}}
\end{equation}
We have
\begin{align*}
   &\P\big[  \wt C_\g \cap  V_{\g,L} \cap  \big\{ M_{T_{n+j_n-1}} >n\rho_0   \big\} \big]
   \le  \P\big[ \wt C_\g \cap  V_{\g,L} \cap  \big\{ S_{\g'}>n\rho_0  \mbox{ for some } |\g'\wedge \g|<n+j_n-L \big\} \big] \\
  &\le p_L\!\!\! \sum_{i=0}^{n+j_n-L}\sum_{k=1}^\8 \P\big[ C_\g \cap \big\{ C_{\g|_i < u- \d(n+j_n - i)}\mbox{ and } S_{\g|_i\g'}>\rho_0 n \mbox{ for some } |\g'|=k \mbox{ and }
  \big|\g\wedge(\g|_i)\g'\big| = i  \big\}
  \big]\\
& \leq p_L\sum_{i = 0}^{n+j_n - L} \sum_{k=1}^{\8} N^{k} \P\left[ C_{\gamma} \right] \P\left[ \S_{k} > \delta(n+j_n - i ) \right]\\
& \leq p_L\P\left[ C_{\gamma} \right] \sum_{i = 0}^{n+j_n - L} \sum_{k=1}^{\8} \psi(\beta)^{k} e^{-\beta \delta(n +j_n- i)}\\
& \leq Cp_L \P\left[ C_{\gamma} \right] \sum_{i = 0}^{n+j_n - L} e^{-\beta \delta(n+j_n - i)}\\
& \leq C p_L \big( e^{-\beta \d} \big)^L \P\left[ C_{\gamma} \right].
\end{align*}

\medskip

\noindent
{\sc Step 2B$_3$.} Estimates \eqref{eq:29}, \eqref{eq:26} and \eqref{eq:27} imply that for some $\kappa_3<1$
\begin{align*}
\P\bigg[ \bigcup_{\g\in U}\wt C_\g \cap \big\{ M_{T_{n+j_n-1}}>n\rho_0   \big\}
\bigg] &\le\frac{ Cp_L \kappa_3^L}{N^L \sqrt n}   e^{-\a_0 \rho_0 n} e^{-\frac{\rho_0^2 j_n^2}{ 2\sigma_0^2 n}}\\
& \frac{ C \kappa_3^L\eps_1(L)}{\sqrt n}
  e^{-\a_0 \rho_0 n} e^{-\frac{\rho_0^2 j_n^2}{ 2\sigma_0^2 n}}
  \end{align*}
 and choosing large $L$ such that $C\kappa_3^L < 1/2$ we conclude \eqref{eq:25} and complete proof of the Lemma.
\end{proof}

\subsection{Proof of Proposition \ref{prop: cz1}}

The proof of the asymptotic behavior relies again on the block decomposition. We focus here on the proof of \eqref{eq: ren44'}.
Fix large $L$  and consider the family of $N^{n+j_n-L}$ disjoint sets $\T_{\g,L}$ for $\gamma \in T_{n+j_n-L}$. Given such a $\g$ define
\begin{equation}\label{eq:bgl}
\begin{split}
B_{\g,L} &= \big\{ M_\g \le n\rho_0, M_{\T_{\g,L-1}}\le n\rho_0, M_{T_{\g,L}} >n\rho_0  \big\}\\
&= \{ \max\limits_{k \leq n+j_n - L} S_{\gamma|_k} \leq n\rho_0, \max\limits_{|\gamma'| \leq L - 1} S_{\gamma \gamma'} \leq n\rho_0, \max\limits_{|\gamma'| = L} S_{\gamma \gamma'} > n\rho_0\}
\end{split}
\end{equation}
This is the event that on the path from $o$ to $\g$ and in the subtree  $\T_{\g,L-1}$ the value $n\rho_0$ is not exceeded, but $S_{\g'}>n\rho_0  $ for some $\g'\in T_{\g,L}\subset T_{n+j_n}$.  The set $T_{\g,L}$ consists of $N^L$ elements and the corresponding paths from the root are dependent. The following Lemma is the main step of the proof.

\begin{lem}\label{prop:1}  There exists $L_0$ such that for any fixed $L>L_0$ and
  any $\g\in T_{n+j_n-L}$
$$ \P(B_{\g,L})
 \sim \frac{c C_L }{\sqrt n N^{n+j_n-L}}   e^{-\a_0 \rho_0 n} e^{-\frac{\rho_0^2 j_n^2}{ 2\sigma_0^2 n}}
$$ as $n\to \8$ uniformly for $|j_n| \le b \sqrt{n\log n}$,
 where $C_L =  \mathbb{E}\Big[\big( e^{\alpha \overline{M}_{L}} - e^{\alpha M_{L-1}} \big)_{+} \Big]$
\end{lem}
\begin{proof} Observe first that for fixed
 $\g\in T_{n+j_n-L}$
we have
\begin{equation*}%\label{eq6}
\begin{split}
 \mathbb{P}\big[ M_{\T_{\g,L-1}}\le n\rho_0, &     M_{T_{\g,L}} > n\rho_0\big]\\
 & = \mathbb{P}\left[ M_{\g} > n\rho_0, M_{\T_{\g,L-1}}\le n\rho_0, M_{T_{\g,L}} >  n\rho_0 \right]  \\
 & \quad + \mathbb{P}\left[ M_\g \le n\rho_0, M_{\T_{\g,L-1}}\le n\rho_0, M_{T_{\g,L}} > n\rho_0 \right] \\
  & = \mathbb{P}\left[ M_\g > n\rho_0, M_{\T_{\g,L-1}}\le n\rho_0, M_{T_{\g,L}} >  n\rho_0 \right] + \mathbb{P}\left[B_{\gamma, L} \right].
\end{split}
\end{equation*}

\noindent
{\sc Step 1.} Notice that
\begin{equation}\label{eq:ren23}
\mathbb{P}\left[ M_\g > n\rho_0, M_{T_{\g,L}} >  n\rho_0 \right]  \le  \frac{ C(\alpha, \beta) \delta^{L}} {\sqrt{n} N^{n+j_n-L}}
  e^{-\a_0 \rho_0 n} e^{-\frac{\rho_0^2 j_n^2}{ 2\sigma_0^2 n}},
\end{equation}
where $\delta = \psi(\beta) < 1$ provided $\beta < \alpha_0$. Indeed, by  Lemma 3.3 in \cite{BM} (appropriately modified), for  some fixed $|\gamma_0'| = L$ we have
\begin{align*}
\mathbb{P}\left[ M_\g > n\rho_0, M_{T_{\g,L}} >  n\rho_0 \right]  &= \mathbb{P}\left[\max\limits_{ k \leq n+j_n - L } S_{\gamma |_k} > n\rho_0, \max\limits_{|\gamma'| = L } S_{\gamma \gamma'} > n\rho_0  \right]\\
 & \leq N^{L} \mathbb{P}\left[ \max\limits_{ k\leq n+j_n - L } S_{\gamma |_k} > n\rho_0 , S_{\gamma \gamma_0'} > n\rho_0  \right]\\
&  =\frac{ C \delta^L}{\sqrt{n} N^{n-L}}   e^{-\a_0 \rho_0 n} e^{-\frac{\rho_0^2 j_n^2}{ 2\sigma_0^2 n}}.
\end{align*}
which proves \eqref{eq:ren23}.

\medskip

\noindent
{\sc Step 2.}
Thus, in order to prove the Proposition, it is sufficient to show that for some large fixed $L$
\begin{equation*}
 \mathbb{P}\left[ M_{\T_{\g,L-1}} \le n\rho_0, M_{T_{\g,L}} >n\rho_0  \right] \sim \frac{C C_L }{\sqrt n N^{n+j_n-L}}   e^{-\a_0 \rho_0 n} e^{-\frac{\rho_0^2 j_n^2}{ 2\sigma_0^2 n}}
\end{equation*}
 We write
\begin{equation}\label{eq:30}
\begin{split}
 \mathbb{P}\big[ M_{\T_{\g,L-1}}&\le n\rho_0, M_{T_{\g,L}} >n\rho_0 \big]=
   \mathbb{P}\left[S_{\gamma} \leq n\rho_0 -n^{\frac{1}{4}}, M_{\T_{\g,L-1}} \leq n\rho_0, M_{T_{\g,L}} >n\rho_0\right]\\
 & + \mathbb{P}\Big[n\rho_0 -n^{\frac{1}{4}} < S_{\gamma} < n\rho_0 , M_{\T_{\g,L-1}} \leq n\rho_0, M_{T_{\g,L}} >n\rho_0, \max\limits_{|\gamma'| = L} S_{ \gamma'} \ge  n^{\frac 14} \Big]\\
 &+\mathbb{P}\Big[n\rho_0 -n^{\frac{1}{4}} < S_{\gamma} < n\rho_0, M_{\T_{\g,L-1}}\leq n\rho_0, M_{T_{\g,L}} >n\rho_0, \max\limits_{|\gamma'| = L} S_{ \gamma'} < n^{\frac 14} \Big]
\end{split}
\end{equation}
We prove that the first two terms are negligible and one needs to consider the asymptotic behavior of the third one.

To estimate the first summand fix $\beta > \a$ and observe that by Markov's inequality with functions $e^{\alpha_0 x}$ and $e^{\beta x}$ for fixed $|\gamma'| = L$ we have
\begin{equation}\label{eq:5.4}
\begin{split}
\mathbb{P}[S_{\gamma} \leq n\rho_0 -n^{\frac{1}{4}}&, M_{T_{\g,L}} >n\rho_0]\\
 &\leq N^L \sum_{m \geq 0} \mathbb{P}\left[ -(m+1) < \S_{n+j_n-L} - (n\rho_0 -n^{\frac{1}{4}})\leq  -m, \S_{n+j_n} >n\rho_0 \right]\\
 & \leq N^L \sum_{m \geq 0} \mathbb{P}\left[\S_{n+j_n-L}  > n\rho_0 -n^{\frac{1}{4}} -(m+1)\right] \mathbb{P}\left[\S_L > n^{\frac{1}{4}}+ m \right] \\
 & \leq N^L \sum_{m \geq 0}  \lambda(\alpha_0)^{n+j_n  - L}  e^{-\alpha_0\rho_0 n}  e^{\alpha_0 n^{\frac{1}{4}}} e^{\alpha_0(m+1)} \lambda(\beta)^{L} e^{-\beta n^{\frac{1}{4}}} e^{-\beta m } \\
 & = e^{-{\alpha_0}\rho_0 n} N^{L-n-j_n}  e^{(\alpha_0 - \beta)n^{\frac{1}{4}}} \lambda(\alpha_0)^{ - L}  \lambda(\beta)^{L} \sum_{m \geq 0} e^{\alpha_0(m+1)} e^{-\beta m }\\ &= o\bigg( \frac{1}{\sqrt{n} N^{n+j_n - L}}   e^{-\a_0 \rho_0 n} e^{-\frac{\rho_0^2 j_n^2}{ 2\sigma_0^2 n}} \bigg).
\end{split}
\end{equation}
The same argument proves
\begin{equation}\label{eq:5.5}
\P\big[ S_{\gamma} > n\rho_0-n^{\frac 14}, \max\limits_{|\gamma'| = L} S_{ \gamma'} > n^{\frac 14}\big] = o\bigg( \frac{1}{\sqrt{n} N^{n+j_n - L}}   e^{-\a_0 \rho_0 n} e^{-\frac{\rho_0^2 j_n^2}{ 2\sigma_0^2 n}}
 \bigg).
\end{equation}
Finally we write \begin{equation}\label{eq8}
\begin{split}
 &\mathbb{P}\Big[n\rho_0-n^{\frac{1}{4}} < S_{\gamma} < n\rho_0 , M_{\T_{\g,L-1}} \leq n\rho_0, M_{T_{\g,L}} >n\rho_0, \max\limits_{|\gamma'| = L} S_{ \gamma'} < n^{\frac 14} \Big] \\
 &=  \int_{0\le y\le x < n^{\frac 14}}  \mathbb{P}\left[n\rho_0 - x < \S_{n+j_n-L} < n\rho_0  - y \right]  \mathbb{P} \left[\max\limits_{|\gamma'| \leq L -1} S_{ \gamma'} \in dy, \max\limits_{|\gamma'| = L} S_{\gamma'} \in dx  \right].
\end{split}
\end{equation}
Now we apply Lemma \ref{lem:petrov2}
$$    \mathbb{P}\left[\S_{n+j_n-L} \geq n\rho_0 - y  \right]  \sim
C(\rho_0) e^{-\a_0 \rho_0 n} N^{-(n+j_n-L)} e^{\a_0 y} e^{-\frac{\rho_0^2 j_n^2}{ 2\sigma_0^2 n}}.
$$
Analogously
$$    \mathbb{P}\left[\S_{n+j_n-L} \geq n\rho_0 - x  \right]  \sim
C(\rho_0) e^{-\a_0 \rho_0 n} N^{-(n+j_n-L)} e^{\a_0 x} e^{-\frac{\rho_0^2 j_n^2}{ 2\sigma_0^2 n}}.
$$
Back to \eqref{eq8} we end up with
\begin{multline*}
 \mathbb{P}\Big[n\rho_0 -n^{\frac{1}{4}} < S_{\gamma} < n\rho_0, S_{\gamma } + \max\limits_{|\gamma'| \leq L -1} S_{ \gamma'} \leq n\rho_0, \max\limits_{|\gamma'| = L} S_{\gamma \gamma'} >n\rho_0, \max\limits_{|\gamma'| = L} S_{ \gamma'} < n^{\frac 14} \Big]\\
 \sim \frac{C(\rho_0)}{\sqrt{n} N^{n +j_n - L}}  e^{-\a_0 \rho_0 n} e^{-\frac{\rho_0^2 j_n^2}{ 2\sigma_0^2 n}} \mathbb{E}\left[\left( e^{\alpha \overline{M}_{L}} - e^{\alpha M_{L-1}} \right)_{+} \right].
\end{multline*}
Note that by the moment assumptions the expectation above  is finite, hence we conclude the Lemma.
\end{proof}

\begin{proof}[{Proof of Proposition \ref{prop: cz1} - formula \eqref{eq: ren44'}}]
 We apply Lemma \ref{prop:1} and proceed similarly as in the proof of Lemma \ref{prop:2} i.e. we divide all the elements of $T_{n+j_n}$ into disjoint sets and use strongly that
 $$
\{ \tau_{n\rho_0} = n +j_n \} \subset \bigcup_{\g \in T_{n+j_n-L}} B_{\g, L},
$$   for $B_{\g,L}$  defined in \eqref{eq:bgl}.

   We will prove that for any $\eps>0$ there
  is a constant $C_L$ depending on the parameter $L$, that will be specified below, and such that
  \begin{equation}\label{eq:31}
\begin{split}
    C_L-\eps & \le \liminf_{n\to\8} \P[\tau_{n\rho_0} = n+j_n] \cdot \sqrt n  e^{\a_0 \rho_0 n} e^{\frac{\rho_0^2 j_n^2}{ 2\sigma_0^2 n}}\\
    &\le \limsup_{n\to\8} \P[\tau_{n\rho_0} = n+j_n] \cdot \sqrt n   e^{\a_0 \rho_0 n} e^{\frac{\rho_0^2 j_n^2}{ 2\sigma_0^2 n}}
    \le C_L + \eps.
\end{split}
  \end{equation}
Then, passing with $\eps$ to 0 and applying Lemma \ref{prop:2}, we obtain
$$ \lim_{n\to\8} \P[\tau_{n\rho_0} = n+j_n] \cdot \sqrt n  e^{\a_0 \rho_0 n} e^{\frac{\rho_0^2 j_n^2}{ 2\sigma_0^2 n}} = C,
$$ for some $C\in (0,\8)$.

Our proof consists of two steps. First we prove that
\begin{equation}\label{eq:833}
  \bigg| \P\bigg[ \bigcup_{\g\in T_{n+j_n-L}} B_{\g,L} \bigg] - \sum_{\g\in T_{n+j_n-L}} \P(B_{\g,L})
   \bigg| \le \frac{\eps}{2\sqrt n}  e^{-\a_0 \rho_0 n} e^{-\frac{\rho_0^2 j_n^2}{ 2\sigma_0^2 n}}
\end{equation} which in view of Lemma \ref{prop:1} entails
\begin{equation*}
\begin{split}
    c_L-\eps/2 & \le \liminf_{n\to\8} \P\bigg[\bigcup_{\g\in T_{n+j_n-1}} B_{\g,L}\bigg] \cdot \sqrt n  e^{\a_0 \rho_0 n} e^{\frac{\rho_0^2 j_n^2}{ 2\sigma_0^2 n}}\\
    &\le \limsup_{n\to\8} \P\bigg[ \bigcup_{\g\in T_{n+j_n-1}} B_{\g,L} \bigg]  \cdot \sqrt n  e^{\a_0 \rho_0 n} e^{\frac{\rho_0^2 j_n^2}{ 2\sigma_0^2 n}}
    \le C_L + \eps/2.
\end{split}
\end{equation*}
Next we show
\begin{equation}\label{eq:34}
  \P\bigg[  \bigcup_{\g\in T_{n+j_n-1}} B_{\g,L} \setminus \{\tau_{n\rho_0} = n+j_n\} \bigg] \le  \frac{\eps}{2\sqrt n} e^{-\a_0 \rho_0 n} e^{-\frac{\rho_0^2 j_n^2}{ 2\sigma_0^2 n}}.
\end{equation}
Thus, \eqref{eq:833} and \eqref{eq:34} imply \eqref{eq:31} and the main result.

\medskip

\noindent
{\sc Step 1. Proof of \eqref{eq:833}.} By the inclusion-exclusion formula
\begin{equation*}
\begin{split}
\Bigg | \P \bigg[ \bigcup_{\gamma\in T_{n +j_n- L} } B_{\gamma,L} \bigg] - \sum_{\gamma\in T_{n +j_n- L}} \P \big[ B_{\gamma,L} \big] \Bigg | \leq \sum_{\gamma,\gamma'\in T_{n+j_n - L}, \g\not=\g'} \P \left[ B_{\gamma,L} \cap B_{\gamma',L} \right].
\end{split}
\end{equation*}
Choose $\beta$ such that $\psi(\beta)<1$, then the Markov inequality entails
\begin{align*}
  \sum_{\gamma,\gamma'\in T_{n+j_n - L}, \g\not=\g'} & \P  \left[ B_{\gamma,L} \cap B_{\gamma',L} \right]
  = \!\!\!\sum_{\g\in T_{n+j_n-L}} \!\!\! \sum_{k=0}^{n+j_n-L-1} \!\!\! \sum_{\{ \g'\in T_{n+j_n-L}:\; |\g\wedge \g'| = k \}}\!\!\! \P \left[ B_{\gamma,L} \cap B_{\gamma',L} \right]\\
& \le \sum_{\g\in T_{n+j_n-L}} \sum_{k=0}^{n+j_n-L-1} \sum_{\{ \g'\in T_{n+j_n-L}:\;| \g\wedge \g'| = k \}} \P \Big[ B_{\gamma,L} \cap \big\{
S_{\g'|_k}\le n\rho_0 \mbox{ and } M_{T_{\g',L}}>n\rho_0 \big\} \Big]\\
& \le \sum_{\g\in T_{n+j_n-L}} \sum_{k=0}^{n+j_n-L-1} N^{n+j_n-L-k}N^L \P[ B_{\gamma,L}] \P\big[ \S_{n+j_n-k} > 0 \big]\\
& \le \sum_{\g\in T_{n+j_n-L}}   \P[ B_{\gamma,L}] \cdot      \sum_{k=0}^{n+j_n-L-1} \psi(\beta)^{n+j_n-k}\\
& \le C \psi(\beta)^L  \cdot  \sum_{\g\in T_{n+j_n-L}}   \P[ B_{\gamma,L}]
\end{align*} which proves \eqref{eq:833}.

\medskip

\noindent
{\sc Step 2. Proof of \eqref{eq:34}.} We proceed as in the proof of Lemma \ref{prop:2}. Recall the definition of $V_{\g,L}$ given in \eqref{eq:02}. Then for some $\g\in T_{n+j_n-L}$
\begin{align*}
    \P\bigg[  \bigcup_{\g\in T_{n+j_n-L}} B_{\g,L} \setminus \{\tau_{n\rho_0}& = n +j_n\} \bigg]
  =     \P\bigg[  \bigcup_{\g\in T_{n+j_n-L}} B_{\g,L} \cap \{M_{n+j_n-1}>n\rho_0\} \bigg]\\
 & =   N^{n+j_n-L}   \P\big[ B_{\g,L} \cap  \{M_{n+j_n-1}>n\rho_0\} \bigg]\\
 & \le  N^{n+j_n-L} \Big(  \P\big[ B_{\g,L} \cap  V^c_{\g,L} \big] +
   \P\big[ B_{\g,L} \cap V_{\g,L} \cap  \{M_{n+j_n-1}>n\rho_0\} \big]\Big).\\
\end{align*}
Reasoning as in the proof of  Lemma \ref{prop:2} (see the proof of \eqref{eq:26} in step 2B$_1$) one can prove
\begin{equation}\label{eq:38}
  \P\big[ B_{\g,L} \cap  V^c_{\g,L} \big] \le C \kappa_1^L \P\big[ B_{\g,L} \big]
\end{equation}
and (see the proof of \eqref{eq:27} in step 2B$_2$)
\begin{equation}\label{eq:39}
  \P\big[ B_{\g,L} \cap V_{\g,L} \cap  \{M_{n+j_n-1}>u\} \big] \le C \kappa_2^L \P\big[ B_{\g,L} \big].
\end{equation}
Combining \eqref{eq:38} with \eqref{eq:39} we obtain \eqref{eq:34}.
\end{proof}

\section{Proofs of Theorems \ref{thm: tuclt} and \ref{thm: tu-ld} }

\subsection{Law of large numbers for $\tau_u$}
The law of large numbers follows essentially from the following Lemma
\begin{lem}\label{lem: pt6}
Assume that $\Psi(\a_0 + \eps)<\8$ for some $\eps>0$. Let $n_1 = n_u - b\sqrt{n_u\log n_u}$ and $n_2 = n_u + b\sqrt{n_u\log n_u}$ for $n_u = u/\rho_0$. Then for any $\delta>0$ one can choose large $b>0$ such that
\begin{equation}\label{eq:ren61}
  \P[\tau_u \le n_1] \le C e^{-\a_0 u} u^{-\d}
\end{equation}
and
\begin{equation}\label{eq:ren62}
  \P[\tau_u \ge n_2] \le C e^{-\a_0 u} u^{-\d}
\end{equation}
\end{lem}
\begin{proof}
  Taking $\gamma$ some fixed elements of $T_{n_1}$ and applying the Taylor expansion
  $$
  \Psi(\a_0+\eps) + \rho\eps + \frac{\eps^2}2\psi''(s)
  $$ for $s\in[\a_0,\a_0+\eps]$,
  we obtain
\begin{align*}
	\P \left[M_{\gamma} > u\right] \leq & \sum_{k=1}^{|\gamma|} \P \left[S_k > u\right] \leq  \sum_{k=1}^{|\gamma|}\lambda(\a_0 + \epsilon)^{k} e^{-(\a_0 + \epsilon)}\\
&\leq C(\epsilon)N^{-|\gamma|} \psi(\a_0 + \epsilon)^{|\gamma|} e^{-(\a_0 + \epsilon)u}\\
&\leq C(\epsilon)N^{-|\gamma|} e^{ |\g|( \eps \rho_0 + c \eps^2) } e^{-(\a_0 + \epsilon)u}.
\end{align*}
  Therefore choosing $\epsilon=\sqrt{\frac{\log n_u}{n_u}}$  we have
  \begin{align*}
    \P[\tau_u\le n_1] & \le  N^{n_1} \P[M_\g > u]\\
    &\le C e^{n_1 (\eps \rho_0 + c \eps^2)} e^{-(\a_0 + \eps)u}\\
    &\le C e^{(n_u - b\sqrt{n_u\log n_u})(\eps\rho_0 + c \eps^2)}e^{-(\a_0 + \eps)u}\\
    &\le C e^{(\eps^2 - b\eps \rho_0)\log n_u} e^{-\a_0 u}\\
    &\le C u^{-\d} e^{-\a_0 u}\\
  \end{align*}
for appropriate large $b$.

The proof of \eqref{eq:ren62} is similar, but instead of $M_\g$ one has to consider  $\{S_{\gamma|k}>u \mbox{ for some } k>n_2\}$ for some infinite path $\gamma$.
\end{proof}

\begin{proof}[Proof of Theorem \ref{thm: tuclt} - LLN]
The result is a direct consequence of even stronger result, namely using \eqref{eq: jo} and Lemma \ref{lem: pt6} we see that for any $\eps>0$ and large $u$
\begin{align*}
  \P\bigg[ \bigg|
 \frac{\tau_u}{u} - \rho_0 \bigg|>\eps  \bigg| \tau_u<\8 \bigg]
 & \le   \P\big[
\tau_u < u(\rho_0 - \eps)\big| \tau_u<\8 \big] +  \P\big[
\tau_u > u(\rho_0 + \eps)\big| \tau_u<\8 \big] \\
 & \le   \P\big[
\tau_u < n_1 \big| \tau_u<\8 \big] +  \P\big[
\tau_u > n_2 \big| \tau_u<\8 \big] \\ &\le C u^{-\d}.
\end{align*}
\end{proof}

\subsection{Central limit theorem for $\tau_u$}
\begin{proof}[Proof of Theorem \ref{thm: tuclt} - CLT]
  Fix $y\in \R$ and choose $b$ as in Lemma \eqref{lem: pt6}. Let $n=u/\rho_0$. Then applying Lemma \ref{lem: pt6}
  \begin{align*}
    \P\bigg[
   \frac{\tau_u - u/\rho_0}{\sigma_0 \rho_0^{-3/2} \sqrt u} \le y
     \bigg] &\sim \P\big[ n - b\sqrt{n\log n} < \tau_{\rho_0 n} \le n + \sigma_0 \rho_0^{-1} \sqrt n y
     \big]\\
     &= \sum_{-b\sqrt{n\log n} < j \le \sigma_0 \rho_0^{-1}\sqrt n y} \P\big[ \tau_{\rho_0 n} = n+j \big]\\
     &\sim C e^{-\a_0 \rho_0 n}  \sum_{-b\sqrt{n\log n} < j \le \sigma_0 \rho_0^{-1}\sqrt n y}  \frac 1 {\sqrt n} e^{-\frac{\rho_0^2}{2\sigma_0^2}\big(\frac{j}{\sqrt n}\big)^2}
  \end{align*}
  Observe that the last expression is just the Riemann sum, thus
  $$
    \P\bigg[
   \frac{\tau_u - u/\rho_0}{\sigma_0 \rho_0^{-3/2} \sqrt u} \le y
     \bigg]  \sim C e^{-\a_0 \rho_0 n} \int_{-\8}^{\frac{\sigma_0 y}{\rho_0}}
     e^{-\frac 12 \big(\frac{\rho_0 s}{\sigma_0 }\big)^2} ds = C'  e^{-\a_0 u} \Phi(y).
  $$
\end{proof}

\subsection{Large deviations of $\tau_u$}

\begin{lem}\label{lem: ren45}
If $\rho > \rho_0$, then
\begin{equation}\label{eq: ren46}
  \P\big[ \ov M_{n_u - D \log n_u} > u \big]=
  \P[M_k > u \mbox{ for some } k \le n_u - D \log n_u]
   = o\bigg( \frac{1}{\sqrt u} e^{- \frac{\Psi^*(\rho)}{\rho} u} \bigg).
\end{equation}
for appropriately large constant $D$.

If $\rho < \rho_0$, then
\begin{equation}\label{eq: ren47}
  \P[M_k > u \mbox{ for some } k > n_u + D \log n_u]
 = o\bigg( \frac{1}{\sqrt u}e^{- \frac{\Psi^*(\rho)}{\rho} u} \bigg).
\end{equation}
\end{lem}
\begin{proof}
  We prove only the first part of the Lemma i.e. \eqref{eq: ren46}, the second one requires similar arguments. By the Markov inequality for $\eps>0$ we have
  \begin{align*}
    \P\big[ \ov M_{n_u - D \log n_u} > u \big] & \le  \sum_{k=1}^{n_u-D\log n_u} \P[M_k > u]
    \le \sum_{k=1}^{n_u-D\log n_u}  \sum_{\g\in T_k} \P[S_{\g} > u]\\
    &\le \sum_{k=1}^{n_u-D\log n_u} N^k   \P[\S_{k} > u]
 \le \sum_{k=1}^{n_u-D\log n_u} N^k \lambda(\a+\eps)^k e^{-(\a+\eps)u} \\
 &\le \sum_{k=1}^{n_u-D\log n_u} e^{\Psi(\a+\eps)k} k e^{-(\a+\eps)u}.
  \end{align*}
  Now we expand $\Psi$ into a Taylor series
  $$
  \Psi(\a+\eps) = \Psi(\a) + \rho \eps + \frac{\eps^2}2 \Psi''(s)
  $$ for some $s\in[\a+\a+\eps]$ and choose $\eps = 1/\sqrt u$. Then
 \begin{align*}
    \P\big[ \ov M_{n_u - D \log n_u} > u \big]
 &\le  n_u    e^{(\Psi(\a) + \rho \eps + c\eps^2)({n_u-D\log n_u})}  e^{-(\a+\eps)u}\\
 &\le C e^{-\a u} \cdot n_u^{1-D \Psi(\a)} = o\bigg( \frac{1}{\sqrt u} e^{- \frac{\Psi^*(\rho)}{\rho} u}  \bigg)
  \end{align*}
  for appropriately large $D$.
\end{proof}
\begin{proof}[Proof of Theorem \ref{thm: tu-ld}]
First we prove \eqref{eq: ren43}. Take $\Theta(u) = n_u-\lfloor n_u\rfloor$. Then by Proposition \ref{prop: cz1}

\begin{align*}
\P\big[ \tau_u = \lfloor n_u\rfloor \big] &=  \P\big[ \tau_{\lfloor n_u\rfloor \rho + \Theta(u)\rho}  = \lfloor n_u\rfloor  \big]\\
 &\sim \frac{ C e^{-\a \Theta(u)\rho} e^{ {\Psi^*(\rho)} \lfloor n_u\rfloor}}{\sqrt{\lfloor n_u\rfloor \rho}}\\
 &\sim  \frac{C \psi(\a)^{-\Theta(u)} } {\sqrt u} {  e^{- \frac{\Psi^*(\rho)}{\rho} u}}\\
\end{align*}

Now, Lemma \ref{lem: ren45} and Proposition \ref{prop: cz1} entail
\begin{align*}
\P[\tau_u < n_u] &\sim \P\big[ n_u - D\log n_u < \tau_u < n_u  \big]\\
&= \sum_{j=0}^{D\log n_u} \P\big[ \tau_u = \lfloor n_u \rfloor - j  \big]\\
&= \sum_{j=0}^{D\log n_u} \P\big[ \tau_{(\lfloor n_u \rfloor - j)\rho + (\Theta(u)+j)\rho} = \lfloor n_u \rfloor - j \big]\\
&\sim  \sum_{j=0}^{D\log n_u} C(\rho) e^{-\a (\Theta(u) +  j) \rho}
\frac{e^{- {\Psi^*(\rho)} (\lfloor n_u \rfloor - j)}}{\sqrt{\rho(\lfloor n_u \rfloor - j)}}\\
&\sim \frac{ C(\rho) \psi(\a)^{-\Theta(u)}} {\sqrt u} {e^{- \frac{\Psi^*(\rho)}{\rho} u}}
\sum_{j=0}^{D\log n_u} \psi(\a)^{-j}\\
&\sim  \frac{\psi(\a)C(\rho) }{\psi(\a)-1}  \psi(\a)^{-\Theta(u)} \frac{1}{\sqrt u} {e^{- \frac{\Psi^*(\rho)}{\rho} u}},
\end{align*}
which proves \eqref{eq: ren41}. Analogous arguments can be used to prove \eqref{eq: ren42}. We omit details
\end{proof}


\begin{thebibliography}{10}




\bibitem{addario}
{\sc L.~Addario-Berry, B.~Reed.}
\newblock Minima in branching random walks.
\newblock {\em Ann. Probab.}, 37,  1044--1079, 2009.

\bibitem{af1}
{\sc V. I. Afanasyev. }
\newblock On the maximum of a subcritical branching process in a random environment.
\newblock {\em  Stochastic Process. Appl. 93, no. 1, 87–107, 2001.}

\bibitem{af2}
{\sc V. I. Afanasyev. }
\newblock High level subcritical branching processes in a random environment.
\newblock {\em  Proc. Steklov Inst. Math. 282, no. 1, 4–14,  2013.}



\bibitem{A}
{\sc E. A\"{i}d\'{e}kon.}
\newblock Convergence in law of the minimum of a branching random walk.
\newblock {\em Ann. Probab.,} 41, 1362--1426, 2013.


%\bibitem{Arf}
%{\sc G. Arfwedson.}
%\newblock Research in collective risk theory. {P}art {I}{I}.
%\newblock {\em Skand. Aktuarietidskr.\/}, 53--100, 1955.

%\bibitem{Asmussen}
%{\sc S. Asmussen.}
%\newblock {\em Ruin Probabilities}.
%\newblock River Edge, NJ: World Scientific, 2000.

\bibitem{BR}
{\sc R.~Bahadur, R.~Rango Rao.}
\newblock On deviationos of the sample mean.
\newblock {\em Ann. Math. Statist.,} 31, 1015--1027, 1960


\bibitem{biggins}
{\sc J.~D.~Biggins.}
\newblock {The first- and last-birth problems for amultitype age-dependent branching
process. }
\newblock {\em Adv. in Appl. Probab.,}  8, 446--459, 1976.



\bibitem{BCDZ}
{\sc D. Buraczewski, J. F. Collamore, E. Damek, J. Zienkiewicz.}
\newblock Large deviation estimates for exceedance times of perpetuity sequences and their dual processes.
\newblock {\em Ann. Probab., }  44(6), 3688--3739, 2016.



\bibitem{BDZ} {\sc D. Buraczewski, E. Damek, J. Zienkiewicz.}
\newblock Pointwise estimates for first passage times of perpetuity sequences.
\newblock Preprint, arxiv.org/abs/1512.03449.

\bibitem{BDZ2} {\sc D. Buraczewski, E. Damek, J. Zienkiewicz.}
\newblock Precise tail asymptotics of fixed points of the smoothing transform with general weights.
\newblock {\em Bernoulli,}  21(1), 489--504, 2015.



\bibitem{BD} {\sc. D. Buraczewski, P. Dyszewski.}
\newblock Large deviation estimates for branching process in random environment.
\newblock Preprint, arxiv.org/abs/1706.03874.


\bibitem{BM}
\textsc{D.~Buraczewski, M.~Ma\'slanka.}
\newblock{Precise large deviations for the first passage time of random walk with negative drift}
\newblock{\em Proc. AMS, to appear.}

%\bibitem{CR}
%\textsc{B. Chauvin, A. Rouault.}
%\newblock{KPP equation and supercritical branching Brownian motion in the subcritical speed area. Application to spatial trees.}
%\newblock{\em  Probab. Theory Related Fields 80, no. 2, 299--314, 1988.}





%\bibitem{C}
%{\sc H. Cram\'er.}
%\newblock On the mathematical theory of risk.
%\newblock {\em Skandia Jubilee Volume}, Stockholm, 1930.


\bibitem{DZ}
{\sc A. Dembo, O.~Zeitouni.}
\newblock {\em Large Deviations Techniques and Applications}.
\newblock Boston: Jones and Bartlett, 1993


%\bibitem{DS1} {\sc B. Derrida, Z. Shi.}
%\newblock Large deviations for the rightmost position in a branching Brownian motion.
%\newblock Preprint, arxiv.org/abs/1702.08505.

%\bibitem{DS2} {\sc B. Derrida, Z. Shi.}
%\newblock Slower deviations of the branching Brownian motion and of branching random walks.
%\newblock Preprint, arxiv.org/abs/1705.02277.



%\bibitem{F}
%{\sc W. Feller.}
%\newblock An introduction to probability theory and its applications.
%\newblock John Wiley and Sons Inc. New York, 1966.

%\bibitem{GrubelK}
%{\sc R.~Gr\"ubel, Z.~Kabluchko.}
%\newblock Edgeworth expansions for profiles of lattice branching random walks.
%\newblock Preprint, arXiv:1503.04616.


\bibitem{hammersley}
{\sc J.~M.~Hammersley.}
\newblock  Postulates for subadditive processes.
\newblock {\em Ann. Probab.,} 2, 652--680, 1974.



\bibitem{hoglund}
{\sc T. H\"oglund.}
\newblock An Asymptotic Expression for the Probability of Ruin within Finite Time.
\newblock {\em Ann. Probab.}, 18,  378--389, 1990.


\bibitem{Hu:Shi}
{\sc Y.~Hu, Z.~Shi.}
\newblock Minimal position and critical martingale convergence in branching
random walks, and directed polymers on disordered trees.
\newblock {Ann. Probab.,}  37, 742--789, 2009.


\bibitem{JO}
{\sc  P. Jelenkovic, M. Olvera-Cravioto.}
\newblock Maximums on trees.
\newblock {Stochastic Process. Appl. 125, no. 1, 217--232, 2015.}



\bibitem{kingman}
{\sc J.~F.~C.~Kingman.}
\newblock The first birth problem for an age-dependent branching process.
\newblock {\em Ann. Probab.} 3, 790--801, 1975.



\bibitem{Lalley}
{\sc S. Lalley.}
\newblock Limit theorems for first-passage times in linear and nonlinear renewal theory.
\newblock {\em Adv. in Appl. Probab.} 16, no. 4, 766--803, 1984.

%\bibitem{M}
%{\sc T. Madaule.}
%\newblock The tail distribution of the Derivative martingale and the global minimum of the branching random walk.
%\newblock Preprint, arxiv.org/abs/1606.03211.



%\bibitem{maslanka}
%\textsc{M.~Ma\'slanka.}
%\newblock{\ldots}
%\newblock{\em .}



\bibitem{Petrov} {\sc V. Petrov.}
\newblock On the probabilities of large deviations for sums of independent random variables.
\newblock {\em Theory Probab. Appl.}  10, 287--298, 1965.

\bibitem{R} {\sc A. Rouault.}
\newblock Precise estimates of presence probabilities in the branching random walk.
\newblock {\em Stochastic Process. Appl.}  44(1), 27--39, 1993.

\bibitem{Shi}
{\sc Z. Shi.}
\newblock {\em Branching random walks}.
\newblock Springer, 2015.



%\bibitem{S}
%{\sc D. Siegmund.}
%\newblock Corrected diffusion approximations in certain random walk problems.
%\newblock {\em Adv. Appl. Prob.} 11, 701--719, 1979.

\bibitem{von}
{\sc von Bahr.}
\newblock Ruin probabilities expressed in terms of ladder height distributions.
\newblock {\em Scand.
Actuar. J. 190--204, 1974.}



\end{thebibliography}
\end{document}